\documentclass[12pt, oneside]{amsart}   
\allowdisplaybreaks
\usepackage{hyperref}
\usepackage{geometry,comment}                		% See geometry.pdf to learn the layout options. There are lots.
\usepackage{graphicx}				% Use pdf, png, jpg, or eps§ with pdflatex; use eps in DVI mode
\usepackage{amssymb}

\newtheorem{theorem}{\bf Theorem}[section]
\newtheorem{proposition}[theorem]{\bf Proposition}
\newtheorem{lemma}[theorem]{\bf Lemma}
\newtheorem{corollary}[theorem]{\bf Corollary}

\newtheorem{example}[theorem]{\bf Example}

\newtheorem{remark}[theorem]{\bf Remark}

\newenvironment{proofofthm}[1]{\noindent{\it Proof of Theorem#1}}{\hfill$\square$\\\mbox{}}

\def\coor{{\mathcal{O}}}
\def\mc{{\mathbb{C}}}
\def\mn{{\mathbb{N}}}
\def\mz{{\mathbb{Z}}}
\def\ideal{{\mathcal{I}}}
\def\een{{\mathrm{End}}}
\def\uuu{{\mathcal{U}}}
\def\OR{{\mathrm{O}}}
\def\SP{{\mathrm{Sp}}}
\def\SO{{\mathrm{SO}}}
\def\SL{{\mathrm{SL}}}
\def\GL{{\mathrm{GL}}}
\def\cone{\mathcal{C}}
\title[On the cones of classical groups]
{On the cones of classical groups}
\author[M. Domokos]{M\'aty\'as Domokos}
\address{HUN-REN Alfr\'ed R\'enyi Institute of Mathematics,
Re\'altanoda utca 13-15, 1053 Budapest, Hungary,
ORCID iD: https://orcid.org/0000-0002-0189-8831}
\email{domokos.matyas@renyi.hu}
\thanks{Partially supported by the Hungarian National Research, Development and Innovation Office,  NKFIH K 138828.}

\begin{document}
\maketitle
%\begin{center} \dedicatory{Dedicated to Vesselin Drensky on his 75th birthday}\end{center}
\begin{abstract} 
The cone of a classical group $G$ is an affine $G\times G$-variety. The aim of this note is to initiate its combinatorial study 
in the cases when $G$ is the complex orthogonal or symplectic group.  
%(the case of the general linear group being well documented in the literature)
The coordinate ring of the cone of $G$  is a finitely generated commutative graded algebra. 
First the $G\times G$-module structure of its homogeneous components is determined.  
This is used to compute the Hilbert series of this coordinate ring in the cases when $G$ is the orthogonal group $\OR(3)$, $\OR(4)$, the special orthogonal group 
$\SO(4)$, and when $G$ is the symplectic group 
$\SP(4)$. It is concluded that the coordinate ring of the cone of $\OR(3)$ is not Koszul, hence the vanishing ideal of this cone has no quadratic Gr\"obner basis 
(although it is minimally generated by quadratic elements). 
\end{abstract}

\section{Introduction}

Quantum groups were originally defined as Hopf algebra deformations of  universal enveloping algebras of complex semisimple Lie algebras. 
A dual approach was taken by Faddeev, Reshetikhin and Tahtajan \cite{rtf}, who defined quantized coordinate rings of classical groups 
by  taking a noncommutative deformation of the coordinate ring of $G$.  This is a two-step process: first they define a matrix bialgebra (called the FRT-bialgebra), 
and then specialize to $1$ a central group-like element in it.  
Let $G$ be any of the classical groups $\SL(n)$, $\OR(n)$, or $\SP(n)$ (for $n$ even in the latter case). 
The FRT-bialgebra $\mathcal{A}_q(G)$ is an associative graded $\mc$-algebra generated by $n^2$ elements of degree $1$ 
(thought of as the coordinate functions on the space of $n\times n$ "quantum matrices") subject to a set of homogeneous quadratic relations depending on a 
non-zero parameter $q\in \mc$.  The algebra $\mathcal{A}_q(\SL(n))$ is the quantized coordinate ring of $n\times n$ matrices, and  
is thoroughly studied from combinatorial and ring theoretic aspects, see for example \cite{parshall-wang}, \cite{goodearl-lenagan}, \cite{brown-goodearl}, \cite{goodearl}. 
It is a noncommutative deformation of the coordinate ring of the space of $n\times n$ matrices, a polynomial algebra in $n^2$ indeterminates, whose 
$\GL(n)\times \GL(n)$-module structure is given by the Cauchy formula (see for example \cite[Section 11.5.1]{procesi}). 

For $G=\OR(n)$ or $G=\SP(n)$, the algebra $\mathcal{A}_q(G)$ is a deformation of the coordinate ring $\coor(\cone(G))$ 
of the cone $\cone(G)$ of $G$ (see e.g. \cite{domokos-lenagan}). 
Therefore the Hilbert series of the graded algebra $\mathcal{A}_q(G)$ coincides with the Hilbert series of $\coor(\cone(G))$. 
Moreover, the problem of finding a normal form for the elements of $\mathcal{A}_q(G)$ (e.g. 
giving a monomial  basis of $\mathcal{A}_q(G)$ together with the corresponding rewriting rules) is closely related to the similar task for $\coor(\cone(G))$. 
In this paper we initiate the combinatorial study of   $\coor(\cone(G))$. 

We shall work over the base field $\mc$ of complex numbers, and throughout the paper we assume that 
$n\ge 3$ is a positive integer. 
Recall that the orthogonal group is 
\[\OR(n)=\{M\in \mc^{n\times n}\mid M^TM=I\},\] 
the set of $n\times n$ matrices over $\mc$ whose transpose is their inverse. 
For $n=2m\ge 4$ even, the symplectic group is 
\[\SP(n)=\{M\in \mc^{n\times n}\mid M^TJM=J\}, \]
where $J$ is the $n\times n$ matrix obtained by glueing $m$ copies of 
$\begin{pmatrix}0 & 1\\ -1 & 0\end{pmatrix}$ along the main diagonal. 
Now let $G$ be any of $\OR(n)$ or $\SP(n)$. 
Then $G$ is a subset of $\mc^{n\times n}$ defined by polynomial equations. 

By the \emph{cone $\cone(G)$  of $G$} we mean the \emph{Zariski closure} in $\mc^{n\times n}$ of 
\begin{equation*}\label{eq:cone def} 
\{c M\mid c\in  \mc, \quad M\in G\}\subset \mc^{n\times n}.\end{equation*}
Note that the subset of  $\mc^{n\times n}$ in the above line is not Zariski closed in general, see Example~\ref{example:cone}.   
We have the embeddings 
\begin{equation*} \label{eq:embedding} G\hookrightarrow \cone(G)\hookrightarrow \mc^{n\times n}\end{equation*}  
of affine algebraic varieties. Restriction of functions induces the $\mc$-algebra surjections 
\begin{equation}\label{eq:comorphism} 
\coor(\mc^{n\times n})\stackrel{\varphi_1}\longrightarrow \coor(\cone(G))
\stackrel{\varphi_2}\longrightarrow \coor(G)\end{equation}
between the corresponding coordinate rings. 
Here $\coor(\mc^{n\times n})=\mc[x_{ij}\mid 1\le i,j\le n]$ is an $n^2$-variable polynomial algebra, where $x_{ij}$ is the function assigning to $M\in \mc^{n\times n}$ its $(i,j)$-entry. 
It is endowed with the standard grading. Since $\cone(G)$ is a cone, the surjection $\varphi_1$ transports the grading to $\coor(\cone(G))$. 
So denoting by $\coor(\cone(G))_d$ the degree $d$ homogeneous component of $\coor(\cone(G))$ we have 
$\coor(\cone(G))=\bigoplus_{d=0}^\infty \coor(\cone(G))_d$. 
The group $G\times G$ acts on $\mc^{n\times n}$ by 
\[(g,h)\cdot M=gMh^{-1} \ \mbox{ for }g,h\in G \mbox{ and } M\in \mc^{n\times n},\]
with matrix multiplication on the right hand side. The subvarieties $\cone(G)$ and $G$ are both $G\times G$-stable,  
and we have the induced $G\times G$-action on the corresponding coordinate rings. The $\mc$-algebra homomorphisms 
$\varphi_1$ and $\varphi_2$ are $G\times G$-equivariant. 

\subsection{The content of the paper} 
In Section~\ref{sec:peter-weyl} we determine the structure of the representation 
$G\times G$ on $\coor(\cone(G))$ (see Theorem~\ref{thm:O(G)_d GxG-module}). 
Using some known formulae for the dimensions of the irreducible representations of $G$, 
this allows us to deduce in Section~\ref{sec:hilbert series} the Hilbert series of the graded algebra 
$\coor(\cone(G))$ when $G$ is $\OR(3)$, $\OR(4)$, or $\SP(4)$, and the same is done for $\SO(4)$ in Section~\ref{sec:SO hilbert series}. 
Moreover, for the same groups $G$ we compute in Section~\ref{sec:UxU-invariants} also the Hilbert series of 
$\coor(\cone(G))^{U\times U}$, where $U$ is the unipotent radical of a Borel subgroup in $G$. 
The vanishing ideal   
\[\ideal(\cone(G))=\{f\in \coor(\mc^{n\times n})\mid f(c g)=0 \ \forall g\in G, \ \forall c \in \mc \}\] 
of $\cone(G)$ for $G=\OR(n)$ or $G=\SP(n)$  
(i.e. the kernel of $\varphi_1$ from \eqref{eq:comorphism}) was described by  Weyl \cite{weyl},  
who proved that these ideals are generated by explicitly given quadratic elements. 
In Section~\ref{sec:ideal} we present in modern language the proof of this result 
in the case of $\OR(n)$.  
As a consequence of the calculations in Section~\ref{sec:hilbert series} we observe that 
the algebra $\coor(\cone(\OR(3)))$ is not Koszul, hence the ideal $\ideal(\cone(\OR(3)))$ 
has no quadratic Gr\"obner basis.  In Section~\ref{sec:groebner} we 
present a Gr\"obner basis of  $\ideal(\cone(\OR(3)))$; in addition to the $10$ quadratic generators 
it has $5$ cubic elements. 

Several calculations in the paper were performed using the online CoCalc platform \cite{CoCalc}. 
%%%%%%%%%%%%%%%%%%%%%%%%%%%%%%%%%%%%

\section{The $G\times G$-module structure of $\coor(\cone(G))_d$}\label{sec:peter-weyl} 

Throughout this section $G$ stands for $\OR(n)$, or for $\SP(n)$ (where $n=2m$ is even in the latter case).  
The finite dimensional irreducible representations of $G$ (as a  linear algebraic group) are naturally indexed by a set $\Lambda(G)$ of partitions.  
By a \emph{partition} we mean a non-increasing finite sequence of integers. 
Denote by $\mathrm{Par}(n):=\{\lambda=(\lambda_1,\dots,\lambda_n)\mid \lambda_i\in \mn_0,\quad \lambda_1\ge\dots\ge\lambda_n\}$ the set of partitions with 
at most $n$ non-zero parts. 
Recall from \cite{procesi} or \cite{weyl} that we have 
\[\Lambda(\OR(n))=\{\lambda\in \mathrm{Par}(n)\mid |\{i\mid \lambda_i>0\}|+|\{i\mid \lambda_i>1\}|\le n\},\] 
whereas for $n=2m$ even we have 
\[\Lambda(\SP(n))=\mathrm{Par}(m).\] 
Set $\Lambda(G)_d:=\{\lambda\in \lambda(G)\mid \lambda_1+\dots+\lambda_n=d\}$. 
For $\lambda\in \Lambda(G)$ denote by $V_{\lambda}$ the underlying vector space of the irreducible representation of $G$ labeled by $\lambda$, endowed 
with the corresponding linear action of $G$. 
Note that $V_{(1,0,\dots,0)}\cong \mc^n$, endowed with the defininig representation of $G$ (given by multiplication of column vectors by matrices). 
We shall use the following well-known fact (see for example \cite[Section 11.3.1]{procesi}): 

\begin{proposition}\label{prop:tensorpowers} 
For $\lambda\in\Lambda(G)_d$ the $G$-module $V_{\lambda}$ occurs as a summand 
in the $d$th tensor power $\mathrm{T}^d(\mc^n)$ of $\mc^n$, and $V_{\lambda}$ is not a summand of $\mathrm{T}^k(\mc^n)$ 
for $k>d$ with $k-d$ odd,  or 
for $k<d$. 
\end{proposition}

The representation of $G\times G$ on $\coor(G)$ is described by the Peter-Weyl Theorem: 
\begin{proposition}\label{prop:peter-weyl} 
We have the following isomorphism of $G\times G$-modules: 
\[\mathcal{O}(G)\cong \bigoplus_{\lambda\in\Lambda(G)}V_{\lambda}\otimes V_{\lambda}.\]
\end{proposition}

\begin{proof} 
The Peter-Weyl Theorem says that $\mathcal{O}(G)\cong \bigoplus_{\lambda\in\Lambda(G)}V_{\lambda}\otimes V_{\lambda}^*$, 
where $V_\lambda^*$ stands for the dual of the representation $V_\lambda$  (see for example \cite[Theorems 8.3.2 and 8.7.2.3]{procesi}). 
Taking into account the fact that all representations of $G$ are self-dual, i.e. $V_{\lambda}\cong V_{\lambda}^*$ holds for all $\lambda\in\Lambda(G)$, 
we get the desired statement. 
\end{proof}

We mention that a basis of $\coor(\OR(n))$ consisting of so-called bideterminants is given in \cite{cliff}, extending for the orthogonal 
group the analogous theory for $\coor(\mc^{n\times n})$ (cf. \cite[Section 13.4]{procesi}, 
\cite{doubilet-rota-stein}, \cite{deconcini-eisenbud-procesi}). The case of $\coor(\SP(n))$ is treated in \cite{deconcini}. 

\begin{theorem}\label{thm:O(G)_d GxG-module} 
For $d=0,1,2,\dots$ we have the following isomorphisms of $G\times G$-modules: 
\[\coor(\cone(G))_d 
\cong \bigoplus_{k=0}^{\lfloor d/2\rfloor} \bigoplus_{\lambda\in\Lambda(G)_{d-2k}}V_{\lambda}\otimes V_{\lambda}, \]
where $\lfloor d/2\rfloor$ is the lower integer part of $d/2$. 
In particular, for $d\ge 2$ we have 
\[\coor(\cone(G))_d \cong \coor(\cone(G))_{d-2}\oplus \bigoplus_{\lambda\in \Lambda(G)_d} V_{\lambda}\otimes V_{\lambda}.\] 
\end{theorem}
 
\begin{proof} 
Note that $\varphi_2(\varphi_1(x_{ij}))$ are the matrix elements of the defining representation of $G$ on $\mc^n$ 
(the maps $\varphi_1$ and $\varphi_2$ are defined in \eqref{eq:comorphism}). 
Therefore $\varphi_2(\coor(\cone(G))_d)$ is the $\mc$-linear span of the length $d$ products of the matrix elements of the defining representation of $G$, 
hence $\varphi_2(\coor(\cone(G))_d)$ is the space of matrix elements of the $d$-fold tensor power $\mathrm{T}^d(\mc^n)$ of the defining representation of $G$ on $\mc^n$. 
It follows by Proposition~\ref{prop:peter-weyl} that 
$\varphi_2(\coor(\cone(G))_d)$ is isomorphic to $\bigoplus V_{\lambda}\otimes V_{\lambda}$, 
where the direct sum ranges over the $\lambda\in \Lambda(G)$ such that 
$V_{\lambda}$ occurs as a direct summand in $\mathrm{T}^d(\mc^n)$. 
Hence by Proposition~\ref{prop:tensorpowers} we have 
$\varphi_2(\coor(\cone(G))_d)\cong \bigoplus_{k=0}^{\lfloor d/2\rfloor} \bigoplus_{\lambda\in\Lambda(G)_{d-2k}}V_{\lambda}\otimes V_{\lambda}$. 
The restriction to $\coor(\cone(G))_d$ of $\varphi_2$ is injective, since if a homogeneous polynomial in $\coor(\mc^{n\times n})$ vanishes on $G$, then it vanishes on $\cone(G)$ as well. 
Therefore we have $\coor(\cone(G))_d\cong \varphi_2(\coor(\cone(G))_d)$, finishing the proof of the first isomorphism. 
The second isomorphism follows from the first by induction on $d$. 
\end{proof} 

%%%%%%%%%%%%%%%%%%%%%%%%%%%%%%

\section{Some Hilbert series}\label{sec:hilbert series} 

Recall that the \emph{Hilbert series} of a graded algebra $A=\bigoplus_{d=0}^\infty A_d$  
(with finite dimensional homogeneous components $A_d$) is the formal power series 
\[\mathrm{H}(A;t):=\sum_{d=0}^\infty \dim_\mc(A_d)t^d\in \mz[[t]].\] 
For $G=\OR(n)$ or $G=\SP(n)$ introduce the following function $h_G:\mn_0\to\mn_0$: 
\begin{equation}\label{eq:h_G def} 
h_G(d):=\sum_{\lambda\in \Lambda(G)_d} \dim_\mc(V_\lambda)^2. 
\end{equation}

As a consequence of Theorem~\ref{thm:O(G)_d GxG-module} we get the following: 

\begin{corollary}\label{cor:hilbert series general} For $G=\OR(n)$ or $G=\SP(n)$ we have 
\[\mathrm{H}(\coor(\cone(G));t):=\sum_{d=0}^\infty \sum_{k=0}^{\lfloor d/2\rfloor}h_G(d-2k) t^d.\] 
\end{corollary}

\subsection{The case $G=\OR(3)$.} \label{subsec:hilbertseries O(3)}

\begin{theorem}\label{prop:O(3) Hilbert series} 
We have the equality  
\[\mathrm{H}(\coor(\cone(\OR(3)));t)=-t+\sum_{d=0}^{\infty}\binom{2d+3}{3}t^{2d}=
\frac{1+5t+5t^2-6t^3+4t^4-t^5}{(1-t)^4}.\]
\end{theorem}

\begin{proof} We have $\Lambda(\OR(3))_d=\{(d,0,0),(d-1,1,0)\}$, except for $d=3$ or $d=1$, when $\Lambda(\OR(3))_1=\{(1,0,0)\}$ and $\Lambda(\OR(3))_3=\{(3,0,0),(2,1,0),(1,1,1)\}$. 
Moreover, $V_{(d,0,0)}$ can be identified with the space of $3$-variable spherical harmonics of degree $d$ (see for example \cite[Section 5.2.3]{goodman-wallach}), 
and so $\dim_\mc(V_{(d,0,0)})=\binom{d+2}{2}-\binom{d}{2}=2d+1$. Moreover, the restriction to the special orthogonal group  $\SO(3)$ of the representation on $V_{(d-1,1,0)}$ is isomorphic to 
$V_{(d-1,0,0)}$, whence 
$\dim_\mc(V_{(d-1,1,0)})=\dim(V_{(d-1,0,0)})=2(d-1)+1=2d-1$. 
Furthermore, $V_{(1,1,1)}$ can be identified with the third exterior power of $\mc^3$, so 
$\dim_\mc(V_{(1,1,1)})=1$. 
Thus we have 
\[h_{\OR(3)}(d)=\begin{cases} (2d-1)^2+(2d+1)^2=8d^2+2 & \mbox{ for }d\notin \{0,1,3\}; \\ 
1^2 & \mbox{ for }d=0; \\
3^2  & \mbox{ for }d=1; \\ 
1^2+5^2+7^2 & \mbox{ for }d=3.
\end{cases}\]
It follows by Corollary~\ref{cor:hilbert series general} that 
that for $d\ge 2$ or $d=0$ we have 
\[\dim(\coor(\cone(\OR(3)))_d)=\begin{cases}1+3^2+5^2+\dots+(2d+1)^2=\binom{2d+3}{3} & \mbox{ for }d\neq 1;\\ 
1 & \mbox{ for }d=1.
\end{cases}\] 
This shows 
$\mathrm{H}(\coor(\cone(\OR(3)));t)=-t+\sum_{d=0}^{\infty}\binom{2d+3}{3}t^{d}$. 
Using that 
\[\sum_{d=0}^{\infty}\binom{2d+3}{3}t^{2d}=\frac16\left(\frac{d}{dt}\right)^3\left(\sum_{d=0}^{\infty}t^{2d+3}\right)
=\frac 16\left(\frac{d}{dt}\right)^3\frac{t^3}{1-t^2}
=\frac{1+6t^2+t^4}{(1-t^2)^4}\]
we get 
\[\sum_{d=0}^{\infty}\binom{2d+3}{3}t^d=\frac{1+6t+t^2}{(1-t)^4}, \] 
whence  
\[\mathrm{H}(\coor(\cone(\OR(3))),t)=\frac{1+6t+t^2}{(1-t)^4}-t=
\frac{1+5t+5t^2-6t^3+4t^4-t^5}{(1-t)^4}.\] 
\end{proof} 

\begin{proposition}\label{prop:koszul} 
The quadratic algebra $\coor(\cone(\OR(3)))$ is not Koszul. 
\end{proposition} 

\begin{proof} 
For a graded $\mathbb{C}$-algebra $A=\bigoplus_{d=0}^\infty A_d$, 
the \emph{Poincar\'e series} of $A$ is 
\[\mathrm{P}(A,t)=\sum_{d=0}^{\infty} \dim \mathrm{Tor}^A_d(\mathbb{C},\mathbb{C})t^d
=\sum_{d=0}^\infty \dim \mathrm{Ext}^d_A(\mathbb{C},\mathbb{C})t^d.\]
It is well known that if $A$ is Koszul, then the equality 
\begin{equation}\label{eq:HP=1}
\mathrm{H}(\coor(A),-t)\cdot \mathrm{P}(\coor(A),t)=1\end{equation} 
holds in the power series ring $\mz[[t]]$  (see \cite{lofwall},  \cite{froberg}).  
Up to degree $9$ the terms of $1/\mathrm{H}(\coor(\cone(\OR(3))),-t)$ are the same as the corresponding terms of the polynomial 
\[(1+t)^4\left(1+\sum_{d=1}^{9}(5t-5t^2-6t^3-4t^4-t^5)^d\right)\]
which are 
\[- 7330t^9 + 17883t^8 + 11539t^7 + 5129t^6 + 1912t^5 + 628t^4 + 183t^3 + 46t^2 + 9t + 1.\]
Thus the coefficient of $t^{9}$ in $1/\mathrm{H}(\coor(\cone(\OR(3))),-t)$ is negative, hence it can not be a Poincar\'e series. 
Therefore the equality \eqref{eq:HP=1} can not hold for $A=\coor(\cone(\OR(3)))$, 
whence $\coor(\cone(\OR(3)))$ is not Koszul.  
\end{proof} 

\begin{remark} 
Our focus on the Koszul property is motivated by the connection to the construction of certain quantum 
groups (see the Introduction). In fact it was advertised in \cite{manin:1987}, \cite{manin:1988} that quadratic (and in particular,  Koszul) algebras should play a distinguished 
role in quantum group theory. 
 \end{remark} 

\subsection{The case of $\OR(4)$.} 

\begin{theorem} We have the equality  
 \[\mathrm{H}(\coor(\cone(\OR(4)));t)=\frac{1 + 9 t + 27 t^2 + 19 t^3 - 30 t^4 + 34 t^5 - 35 t^6 +21 t^7 - 7 t^8 +t^9}{(1-t)^7}.\] 
\end{theorem} 
 
 \begin{proof} 
For $d=3$ or $d\ge 5$  we have $\Lambda(\OR(4))_d=\{(d-2,1,1,0),\ (d-k,k,0,0) \ \mbox{ for }k=0,\dots,\lfloor d/2\rfloor \}$. For  
$d\le 2$ we have to omit $(d-2,1,1,0)$, whereas for 
$d=4$ the set  $\Lambda(\OR(4))_4$ contains also $(1,1,1,1)$. 
The dimension of $V_\lambda$ for $\lambda\in \Lambda(\OR(4))_d$ is given in the following table. 
Note that the space $V_{(d,0,0,0)}$ can be identified with the space of spherical harmonics of degree $d$ in $4$ variables, 
hence its dimension is $\binom{d+3}{3}-\binom{d+1}{3}$. Moreover, $V_{(d-2,1,1,0)}$ is isomorphic to $V_{(d-2,0,0,0)}$ as an 
$\SO(4)$-module, whereas $V_{(1,1,1,1)}$ can be identified with the $4$th exterior power of $\mc^4$. 
We read off the dimension of $V_{(d-k,k,0,0)}$ for $k>0$ from \cite[Formula 3.28]{samra-king}. 
\[\begin{array}{c||c|c|c|c}
\lambda & (d,0,0,0) & (d-k,k,0,0), 1\le k	\le \lfloor \frac d2\rfloor & (d-2,1,1,0) & (1,1,1,1)	\\ \hline 
\dim_\mc(V_\lambda) & (d+1)^2 & 2(d+1)(d-2k+1) & (d-1)^2	& 1	
	\end{array} \]
It follows that for $h_{\OR(4)}(d)$ defined in \eqref{eq:h_G def} we have 
\[h_{\OR(4)}(d)=\begin{cases} (d+1)^4+(d-1)^4+4\sum_{k=1}^{\lfloor d/2\rfloor}(d+1)^2(d-2k+1)^2 &\mbox{ for }d\notin \{0,2,4\} \\
1 & \mbox{ for }d=0 \\
117 & \mbox{ for }d=2 \\ 
1707 & \mbox{ for }d=4.
\end{cases} \]
So for $d\notin \{0,2,4\}$ we have 
\[h_{\OR(4)}(d)=\frac{2}{3} d^{5} + \frac{10}{3} d^{4} + \frac{32}{3} d^{2} - \frac{2}{3} d + 2 \] 
by Corollary~\ref{cor:hilbert series general}.   
 A routine calculation yields that 
\[\dim(\coor(\cone(\OR(4)))_d)=\begin{cases}  1 & \mbox{ for }d=0 \\
118 & \mbox{ for }d=2 \\ 
1825 & \mbox{ for }d=4 \\ 
p(d) &\mbox{ for }d\notin \{0,2,4\},  
\end{cases} \]
where 
\[p(d):= \frac{1}{18} d^{6} + \frac{2}{3} d^{5} + \frac{20}{9} d^{4} + 4 d^{3} + \frac{85}{18} d^{2} + \frac{10}{3} d + 1.\]
We have $p(0)=1$, $p(2)=119$, $p(4)=1825$, 
so 
\begin{align} \label{eq:dim(coor(cone(O(4)))d} 
\dim_\mc(\coor(\cone(\OR(4)))_d)=\begin{cases}p(d) & \mbox{ for }d\neq 2 \\
p(d)-1 & \mbox{ for }d=2. \end{cases}
\end{align}
We can rewrite $p(d)$ as a linear combination of binomial coefficients as follows:   
\begin{equation}\label{eq:binom O(4)}
p(d)=\binom{d+1}{1}-14\binom{d+2}{2}+34\binom{d+3}{3}-60\binom{d+5}{5}+40\binom{d+6}{6}.
\end{equation}
It follows from \eqref{eq:dim(coor(cone(O(4)))d} and \eqref{eq:binom O(4)} that 
\[\mathrm{H}(\coor(\cone(\OR(4)));t)=-t^2+\frac{1}{(1-t)^2}- \frac{14}{(1-t)^3}+\frac{34}{1-t)^4}-\frac{60}{(1-t)^6}+\frac{40}{(1-t)^7},\]
which gives the formula in our statement. 
\end{proof}  

 \subsection{The case of $\SP(4)$} 

\begin{theorem} We have the equality 
 \[\mathrm{H}(\coor(\cone(\SP(4)));t)=\frac{1+5t+5t^2+t^3}{(1-t)^{11}}.\]
\end{theorem}

\begin{proof} We have $\Lambda(\SP(4))_d=\{(d-k,k)\mid k=0,1,\dots,\lfloor d/2\rfloor\}$,  
and by \cite[Formula 3.29]{samra-king}, 
\[\dim_\mc(V_{(d-k,k)})=\frac{1}{6}(d+3)(d-k+2)(k+1)(d-2k+1).\] 
It follows that 
\[h_{\SP(4)}(d)=
\frac{1}{15120} d^{9} + \frac{1}{560} d^{8} + \frac{3}{140} d^{7} + \frac{3}{20} d^{6} + \frac{97}{144} d^{5} + \frac{481}{240} d^{4} + \frac{29683}{7560} d^{3} + \frac{4069}{840} d^{2} + \frac{473}{140} d + 1\] 
(see \eqref{eq:h_G def} for the definition of $h_G(d)$). 
By Corollary~\ref{cor:hilbert series general} we obtain 
 \begin{align*} \dim_\mc(\coor(\cone(\SP(4)))_d)=\frac{1}{302400} d^{10} + \frac{1}{7560} d^{9} + \frac{47}{20160} d^{8} + \frac{1}{42} d^{7} + \frac{2243}{14400} d^{6} 
 \\ + \frac{49}{72} d^{5} + \frac{121279}{60480} d^{4} + \frac{2963}{756} d^{3} + \frac{121883}{25200} d^{2} + \frac{709}{210} d + 1 
 \\ =-\binom{d+7}{7}+8\binom{d+8}{8}-18\binom{d+9}{9}+12\binom{d+10}{10}.
 \end{align*}
 It follows that 
 \[\mathrm{H}(\coor(\cone(\SP(4)));t)=\frac{-1}{(1-t)^8}+\frac{8}{(1-t)^9}-\frac{18}{(1-t)^{10}}+\frac{12}{(1-t)^{11}},\] 
 giving the formula in our statement. 
\end{proof} 

%%%%%%%%%%%%%%%%%%%%%%%%%
\section{The cone of the special orthogonal group}\label{sec:SO hilbert series} 

The special orthogonal group is the index two subgroup in $\OR(n)$ defined as 
\[\SO(n)=\{M\in \OR(n)\mid \det(M)=1\}.\] 
For $n$ odd the scalar matrix $-I$ belongs to $\OR(n)\setminus \SO(n)$, hence the cone 
of $\SO(n)$ coincides with the cone of $\OR(n)$. 
For $n=2m$ even both scalar orthogonal matrices $I$ and $-I$ have determinant $1$, hence  $\{cM\mid M\in \SO(2m),\ c\in \mc^\times\}$ is disjoint from 
$\{cM\mid M\in\OR(2m)\setminus \SO(2m), \ c\in \mc^\times\}$. In particular, the cone of $\SO(2m)$ is properly contained in the cone of 
$\OR(2m)$. 
The irreducible representations of $\SO(2m)$ are labelled by 
\[\Lambda(\SO(2m))=\{\lambda\in \mz^m\mid (\lambda_1,\dots,\lambda_{m-1},|\lambda_m|)\in \mathrm{Par}(m)\}.\] 
For $\lambda\in \Lambda(\SO(2m))$ write $W_\lambda$ for the corresponding irreducible representation of $\SO(2m)$. 
For $\lambda\in \Lambda(\OR(2m))$ write $V_\lambda\downarrow^{\OR(2m)}_{\SO(2m)}$ for the restriction of the 
$\OR(2m)$-module $V_{\lambda}$ to the subgroup $\SO(2m)$. 
For $\lambda\in \mathrm{Par}(m)$ 
\begin{equation}\label{eq:O-SO branching} 
V_\lambda\downarrow^{\OR(2m)}_{\SO(2m)}\cong  \begin{cases}
W_\lambda & \mbox{ when }\lambda_m=0 \\ 
W_\lambda\oplus W_{\lambda^\circ} & \mbox{ when }\lambda_m>0,\end{cases} \end{equation} 
where $\lambda^\circ:=(\lambda_1,\dots,\lambda_{m-1},-\lambda_m)$  
and we have $\dim_\mc(W_\lambda)=\dim_\mc(W_{\lambda^\circ})$ for all $\lambda\in \mathrm{Par}(m)$  (see e.g. \cite[Page 422]{procesi}). 

Set $\Lambda(\SO(2m))_d:=\{\lambda\in \Lambda(\SO(2m))\mid \lambda_1+\cdots+\lambda_{m-1}+|\lambda_m|=d\}$. 
Proposition~\ref{prop:tensorpowers} implies by \eqref{eq:O-SO branching} the following: 

\begin{proposition}\label{prop:SO-tensorpowers} 
For $\lambda\in\Lambda(\SO(2m))_d$ the $\SO(2m)$-module $W_{\lambda}$ occurs as a summand 
in $\mathrm{T}^d(\mc^{2m})$, and $W_{\lambda}$ is not a summand of $\mathrm{T}^k(\mc^{2m})$ for $k>d$ with 
$k-d$ odd, or for $k<d$. 
\end{proposition}

Repeating verbatim the proof of Theorem~\ref{thm:O(G)_d GxG-module}  
and referring to Proposition~\ref{prop:SO-tensorpowers} (instead of Proposition~\ref{prop:tensorpowers} 
we get the following: 

\begin{theorem}\label{thm:O(SO(2m))_d GxG-module} 
For $d=0,1,2,\dots$ we have the following isomorphisms of $\SO(2m)\times \SO(2m)$-modules: 
\[\coor(\cone(\SO(2m)))_d 
\cong \bigoplus_{k=0}^{\lfloor d/2\rfloor} \bigoplus_{\lambda\in\Lambda(\SO(2m))_{d-2k}}W_{\lambda}\otimes W_{\lambda}. \]
In particular, for $d\ge 2$ we have 
\[\coor(\cone(\SO(2m)))_d \cong \coor(\cone(\SO(2m)))_{d-2}\oplus \bigoplus_{\lambda\in \Lambda(\SO(2m))_d} W_{\lambda}\otimes W_{\lambda}.\] 
\end{theorem}

As a consequence of Theorem~\ref{thm:O(SO(2m))_d GxG-module} we get the following: 

\begin{corollary}\label{cor:SO-hilbert series general} We have the equality 
\[\mathrm{H}(\coor(\cone(\SO(2m)));t):=\sum_{d=0}^\infty \sum_{k=0}^{\lfloor d/2\rfloor}h_{\SO(2m)}(d-2k) t^d,\]
where  
\[h_{\SO(2m)}(d):=\sum_{\lambda\in \Lambda(\SO(2m))_d} \dim_\mc(W_\lambda)^2. \]
\end{corollary}

\subsection{The Hilbert series of $\coor(\cone(\SO(4))$} 

\begin{theorem} 
We have the equality 
\[\mathrm{H}(\coor(\cone(\SO(4));t)=\frac{1+9t+9 t^{2}+t^{3}}{(1-t)^7}.\]
\end{theorem} 

\begin{proof}
We have $\Lambda(\SO(4))_d=
\{(d-k,k),\ (d-k,-k)\mid k=0,1,\dots,\lfloor d/2 \rfloor\}$ 
and 
\[\dim_\mc(W_{(d-k,k)})=\dim_\mc(W_{(d-k,-k)})=(d+1)(d-2k+1).\] 
Consequently, 
\begin{align*}h_{\SO(4)}(d)=(d+1)^4+2\sum_{k=1}^{\lfloor d/2\rfloor} (d+1)^2(d-2k+1)^2
\\ =\frac{1}{3} d^{5} + \frac{5}{3} d^{4} + 4 d^{3} + \frac{16}{3} d^{2} + \frac{11}{3} d + 1.
\end{align*}
By Corollary~\ref{cor:SO-hilbert series general} we get 
\begin{align*} 
\dim_\mc(\coor(\cone(\SO(4))_d)=\frac{1}{36} d^{6} + \frac{1}{3} d^{5} + \frac{29}{18} d^{4} + 4 d^{3} + \frac{193}{36} d^{2} + \frac{11}{3} d + 1
\\ = -\binom{d+3}{3}+12\binom{d+4}{4}-30\binom{d+5}{5}+20\binom{d+6}{6}. 
\end{align*}
It follows that 
\[\mathrm{H}(\coor(\cone(\SO(4)));t)=
\frac{-1}{(1-t)^4}+\frac{12}{(1-t)^5}-\frac{30}{(1-t)^6}+\frac{20}{(1-t)^7},\] 
giving the formula in our statement. 
\end{proof} 

%%%%%%%%%%%%%%%%%%%%%%%%%%%
 
 \section{$U\times U$-invariants}\label{sec:UxU-invariants}  
 
 For $G\in \{\OR(3),\ \OR(4),\ \SP(4),\ \SO(4)\}$, denote by $U$ the unipotent radical of a Borel subgroup of $G$. 
 Recall that $\dim(\coor(\cone(G))_d^{U\times U})$ equals the number of irreducible 
 $G^\circ\times G^\circ$-module direct summands of $\coor(\cone(G))_d$, where 
 $G^\circ$ denotes the identity component of $G$.  Recall that the identity component of $\OR(n)$ is $\SO(n)$, 
 whereas $\SP(n)$ is connected. 
 
 \subsection{$U\times U$-invariants in $\coor(\cone(\OR(3)))$} 
  By the remark above, by  Theorem~\ref{thm:O(G)_d GxG-module} and the description of $\Lambda_d(\OR(3))$ in the first paragraph of Section~\ref{subsec:hilbertseries O(3)} 
 (recall that each irreducible $\OR(3)$-module restricts to an  irreducible $\SO(3)$-module) we have  
 \[\dim(\coor(\cone(\OR(3)))_d^{U\times U})=\begin{cases} d+1& \mbox{ for }d\ge 2\mbox{ or }d=0\\ 1&\mbox{ for }d=1\end{cases}.\] 
 It follows that 
 \[\mathrm{H}(\coor(\cone(\OR(3)))^{U\times U};t)=\frac{d}{dt}\left(\frac{1}{1-t}\right)-t
 =\frac{1-t+2t^2-t^3}{(1-t)^2}
 =\frac{1+t^2+t^3-t^4}{(1-t)(1-t^2)}.\] 
 
 By a linear change of coordinates we may switch  to the orthogonal group preserving the non-degenerate quadratic form associated to the 
symmetric matrix 
\[\beta:=\left(\begin{array}{ccc}0 & 0 & 1 \\0 & 1 & 0 \\1 & 0 & 0\end{array}\right), \mbox{ so consider }\] 
\[\OR(3,\beta)=\{M\in \mathbb{C}^{3\times 3}\colon M^T\beta M=\beta\}, \quad 
\SO(3,\beta)=\{M\in \OR(3,\beta)\mid \det(M)=1\}. \]

 The Lie algebra of a Borel subgroup $B$ in $\SO(3,\beta)$ is spanned by 
 $E_{11}-E_{33}$, $E_{12}-E_{23}$, where $E_{ij}$ stands for the matrix unit having entry $1$ in the $(i,j)$ position and zero 
 in all other positions. The element $E_{11}-E_{33}$ spans the Lie algebra of the maximal torus 
 \[\mathbb{T}=\{\left(\begin{array}{ccc}z & 0 & 0 \\0 & 1 & 0 \\0 & 0 & z^{-1}\end{array}\right)\mid z\in\mathbb{C}^\times\}\]
 in   $\SO(3,\beta)$,  
 and the Lie algebra of the unipotent radical $U$ is spanned over $\mathbb{C}$ by 
 $E_{12}-E_{23}$. 
 Write $x_{ij}$ for the standard coordinate functions on $\mathbb{C}^{3\times 3}$ (and also for their restriction to the cone of $ \OR(3,\beta)$), and 
 set 
 \[X=\left(\begin{array}{ccc}x_{11} & x_{12} & x_{13} \\x_{21} & x_{22} & x_{23} \\x_{31} & x_{32} & x_{33} \end{array}\right).\]
 For a square matrix $M$ with entries in a commutative ring, we write $\det (M)$  and $\mathrm{Tr}(M)$ for the determinant of $M$ and 
 for the trace (the sum of the diagonal entries) of $M$. 
 Explicit computation yields the following bases in $\coor(\cone(G))_d^{U\times U}$ for $d=1,2,3$: 
 \begin{align*} 
 \coor(\cone( \OR(3,\beta)))_1^{U\times U}&=\mathrm{Span}_{\mathbb{C}}\{x_{31}\} \\
 \coor(\cone( \OR(3,\beta)))_2^{U\times U}&=\mathrm{Span}_{\mathbb{C}} \{x_{31}^2, \quad x_{31}x_{22}-x_{21}x_{32},\quad \mathrm{Tr}(X^T\beta X\beta)\} \\
 \coor(\cone( \OR(3,\beta)))_3^{U\times U}&=\mathrm{Span}_{\mathbb{C}}\{x_{31}^3, \quad x_{31}(x_{31}x_{22}-x_{21}x_{32}),\quad x_{31}\mathrm{Tr}(X^T\beta X\beta), 
 \quad \det(X) \}
 \end{align*} 
 So up to degree $3$, the algebra $\coor(\cone(\OR(3)))^{U\times U}$ is generated by $x_{31}$, $x_{31}x_{22}-x_{21}x_{32}$, 
 $\mathrm{Tr}(X^T\beta X\beta)$, $\det(X)$.  
The above $U\times U$-invariants are in fact invariant as elements in $\coor(\mc^{n\times n})$. They are   
$\mathbb{T}$-eigenvectors, so they are highest weight vectors in the $\SO(3,\beta)\times \SO(3,\beta)$-module $\coor(\mc^{n\times n})$.

 \subsection{$U\times U$-invariants in $\coor(\cone(\OR(4)))$} 
 
 \begin{proposition}
 We have 
 \[\mathrm{H}(\coor(\cone(\OR(4))^{U\times U};t)
 =\frac{1+ 3 t^{2} + t^{3} + t^{4} - 2 t^{5} - t^{6}+ t^{7}  }{(1-t)(1-t^2)^2}.\]
 \end{proposition} 
 \begin{proof} 
For $\lambda\in \Lambda(\OR(4))$ denote by $\epsilon(\lambda)\in \{1,2\}$ the number of irreducible $\SO(4)$-module direct summands in the 
irreducible $\OR(4)$-module $V_\lambda$, hence $\varepsilon(\lambda)^2$ is the number of irreducible $\SO(4)\times \SO(4)$-module direct summands in the 
irreducible $\OR(4)\times \OR(4)$-module $V_\lambda\otimes V_\lambda$. 
Set $N_{\OR(4)}(d):=\sum_{\lambda\in \Lambda(\OR(4))_d}\epsilon(\lambda)^2$.  
By Theorem~\ref{thm:O(G)_d GxG-module} we have 
 \[\dim_\mc(\coor(\cone(\OR(4)))^{U\times U}_d)=\sum_{k=0}^{\lfloor d \rfloor}N_{\OR(4)}(d-2k).\] 
Now 
\[N_{\OR(4)}(d)=\begin{cases} 2+4\lfloor d/2\rfloor &\mbox{ if }d\ge 5 \mbox{ or }d=3 \\
1 &\mbox{ if }d=0,1 \\
5 &\mbox{ if }d=2 \\
11 &\mbox{ if }d=4.
\end{cases}\] 
We conclude that for $d=2f$ or $d=2f+1$ with $d\ge 5$ we have 
\[\dim_\mc(\coor(\cone(G))^{U\times U}_d)=2 f^{2} + 4 f + 1=-\binom{f}{0}-2\binom{f+1}{1}+4\binom{f+2}{2},\] 
and one can check that in fact the same holds except for $d=2$, for which 
$2 f^{2} + 4 f + 1\vert_{f=1}-\dim_\mc(\coor(\cone(G))^{U\times U}_d)=7-6=1$. 
It follows that 
\[\mathrm{H}(\coor(\cone(\OR(4))^{U\times U};t)=-t^2+\left(\frac{-1}{1-t^2}-\frac{2}{(1-t^2)^2}+\frac{4}{(1-t^2)^3}\right)(1+t),\] 
from which our statement easily follows. 
\end{proof} 

\subsection{$U\times U$-invariants in $\coor(\cone(\SO(4)))$} 
 
 \begin{proposition}
 We have 
 \[\mathrm{H}(\coor(\cone(\SO(4))^{U\times U};t)
 =\frac{1+  t^{2}}{(1-t)(1-t^2)^2}.\]
 \end{proposition} 
 \begin{proof} 
By Theorem~\ref{thm:O(SO(2m))_d GxG-module} we have 
 \[\dim_\mc(\coor(\cone(\SO(4)))^{U\times U}_d)=\sum_{k=0}^{\lfloor d/2 \rfloor}|\Lambda(\SO(4))_{d-2k}|.\]
 Moreover,  $|\Lambda(\SO(4))_d|=1+2\lfloor d/2\rfloor$. 
 It follows that for $d=2f$ or $2f+1$ we have 
 \[\dim_\mc(\coor(\cone(\SO(4)))^{U\times U}_d)=f^{2} + 2 f + 1
 =-\binom{f+1}{1}+2\binom{f+2}{2},\]
 implying in turn that 
 \[\mathrm{H}(\coor(\cone(\SO(4))^{U\times U};t)=\left(\frac{-1}{(1-t^2)^2}+\frac{2}{(1-t^2)^3}\right)(1+t).\] 
 \end{proof}

\subsection{$U\times U$-invariants in $\coor(\cone(\SP(4)))$} 
 
 \begin{proposition}
 We have 
 \[\mathrm{H}(\coor(\cone(\SP(4))^{U\times U};t)
 =\frac{1}{(1-t)(1-t^2)^2}.\]
 \end{proposition} 
 \begin{proof} 
By Theorem~\ref{thm:O(G)_d GxG-module} we have 
 \[\dim_\mc(\coor(\cone(\SP(4)))^{U\times U}_d)=\sum_{k=0}^{\lfloor d/2 \rfloor}|\Lambda(\SP(4))_{d-2k}|.\]
 Moreover,  $|\Lambda(\SP(4))_d|=1+\lfloor d/2\rfloor$. 
 It follows that for $d=2f$ or $2f+1$ we have 
 \[\dim_\mc(\coor(\cone(\SO(4)))^{U\times U}_d)=\frac{1}{2} f^{2} + \frac{3}{2} f + 1
 =\binom{f+2}{2},\]
 implying in turn that 
 \[\mathrm{H}(\coor(\cone(\SO(4))^{U\times U};t)=\frac{1+t}{(1-t^2)^3}.\] 
 \end{proof}

%%%%%%%%%%%%%%%%%%%%%%%%%%%%%%%%%%%%%%%%

\section{The vanishing ideal of the  cone of the orthogonal group}\label{sec:ideal} 

Recall that for $1\le i,j\le n$ we denote by $x_{ij}\in \coor(\mc^{n\times n})$ the coordinate function mapping a matrix $M\in \mc^{n\times n}$ 
to its $(i,j)$-entry, so $\coor(\mc^{n\times n})$ is the polynomial algebra $\mc[x_{ij}\mid 1\le i,j\le n]$.  
Denote by $X$ the $n\times n$ generic matrix whose $(i,j)$-entry is $x_{ij}$. 
Write $X^T$ for the transpose of $X$, and for an $n\times n$ matrix $M$ over $\coor(\mc^{n\times n})$ write 
$M_{ij}$ for the $(i,j)$-entry of $M$. 
The following statement is due to H. Weyl: 

\begin{theorem}\label{thm:ideal_generators} {\rm \cite[Chapter V.A.4]{weyl}}
The ideal $\ideal(\cone(\OR(n)))$ is generated by the following $n^2+n-2$ quadratic polynomials: 
\begin{align*} (X^TX)_{ij},\quad (XX^T)_{ij} \quad \text{for } 1\le i<j\le n, 
\\ (X^TX)_{11}-(X^TX)_{kk}, \quad (XX^T)_{11}-(XX^T)_{kk} \quad \text{for }k=2,\dots,n
\end{align*}
\end{theorem} 

\begin{remark} (i) Clearly, a minimal homogeneous generating set of the ideal $\ideal(\cone(\OR(n))$ spans an $\OR(n)\times \OR(n)$-invariant subspace in 
$\coor(\mc^{n\times n})$. 
The $\OR(n)\times \OR(n)$-module structure of the degree $2$ homogeneous component of $\ideal(\cone(\OR(n))$ 
is given (using the notation of Section~\ref{sec:peter-weyl}) as follows: 
\[\ideal(\cone(\OR(n))_2\cong V_{(2,0,\dots,0)}\otimes V_{(0,\dots,0)}\oplus V_{(0,\dots,0)}\otimes V_{(2,0,\dots,0)}\] 
(recall that $V_{(0,\dots,0)}$ is the $1$-dimensional trivial $\OR(n)$-module, 
and $V_{(2,0,\dots,0)}$ can be identified with the space of $n$-variable spherical harmonics of degree $2$).  
This shows that the generating set of $\ideal(\cone(\OR(n))$ provided by Theorem~\ref{thm:ideal_generators} is minimal. 

(ii) A statement similar to Theorem~\ref{thm:ideal_generators} holds for the vanishing ideal of the cone of the symplectic group, see \cite[Chapter X.5]{weyl}.  
\end{remark}

\begin{example}\label{example:cone} Denoting by $\mathrm{i}$ the imaginary complex unit with $\mathrm{i}^2=-1$, 
the matrix $A:=\left(\begin{array}{ccc}1 & \mathrm{i} & 0 \\ -\mathrm{i} & 1 & 0 \\ 0 & 0 & 0\end{array}\right)$ 
satisfies $AA^T=0=A^TA$, hence all the generators of $\mathcal{I}(\cone(\OR(3)))$ given in Theorem~\ref{thm:ideal_generators} 
vanish on $A$. It follows that  $A$ belongs to the cone $\cone(\OR(3))$. However, $A$ is not a scalar multiple of an element of $\OR(3)$. 
\end{example} 

For sake of completeness we provide  a  proof  of Theorem~\ref{thm:ideal_generators} in modern language. 
Write $V=\mc^n$, and let $e_1,\dots,e_n$ denote the standard basis vectors. Identify 
$\mc^{n\times n}$ with $\een_\mc(V)$ in the standard way, then $G=\OR(n)$ is identified with the subgroup  of the general linear group $\GL(V)$ consisting of the linear transformations that preserve the non-degenerate symmetric bilinear form $\langle-\mid-\rangle$ on $V$ given by 
\[\langle e_i\mid e_j\rangle=\begin{cases} 1\mbox{ if }i=j\\
0\mbox{ if }i\neq j. 
\end{cases} \]
For a finite dimensional $\mc$-vector space $W$ we write $W^*$ for the dual space of $W$, 
and for a positive integer $d$ we write $\mathrm{T}^d(W)=W\otimes \cdots \otimes W$ for the $d$-fold tensor power of $W$. We shall make use of the standard identifications  
$W\otimes W^*\cong \een_\mc(W)$ and $\mathrm{T}^d(W^*)\cong \mathrm{T}^d(W)^*$. 
These induce an identification 
\[\mathrm{T}^d(\een_\mc(V))\cong  \een_\mc(\mathrm{T}^d(V))\] 
via 
\[\mathrm{T}^d(V\otimes V^*)\cong V\otimes V^*\otimes \cdots  V\otimes V^*
\cong \mathrm{T}^d(V)\otimes \mathrm{T}^d(V^*)\cong \mathrm{T}^d(V)\otimes \mathrm{T}^d(V)^*.\] 
Write $\uuu_d$ for the $\mc$-subspace 
\[\uuu_d=\mathrm{Span}_\mc\{M\otimes \cdots \otimes M \in 
\een_\mc(\mathrm{T}^d(V))\mid M\in \een_\mc(V)\}.\] 
Clearly $\uuu_d$ is a subalgebra of the noncommutative $\mc$-algebra $\een_\mc(\mathrm{T}^d(V))$, called a {\it Schur algebra}. It contains the subalgebra 
\[\uuu(G)_d=\mathrm{Span}_\mc\{g\otimes \cdots \otimes g \in 
\een_\mc(\mathrm{T}^d(V))\mid g\in \OR(n)\}.\] 
Recall that the polynomial ring $\coor(\mc^{n\times n})$ is graded in the usual way, and 
$\coor(\mc^{n\times n})_d$ stands for the degree $d$ homogeneous component. 

\begin{lemma}\label{lemma:schur-weyl} 
There exists a unique  $\mc$-vector space isomorphism 
\[\coor(\mc^{n\times n})_d\to \uuu_d^*, \qquad f\mapsto \tilde f\]  
such that for any $f\in \coor(\mc^{n\times n})_d$ and $M\in \mc^{n\times n}$ we have  
\[\tilde f(M\otimes \cdots \otimes M)=d!f(M).\] 
\end{lemma} 
\begin{proof}  
Consider the isomorphisms 
\[\coor(\mc^{n\times n})_d\cong (\mathrm{T}^d(\een_\mc(V))^*)^{S_d}\cong 
 (\mathrm{T}^d(\een_\mc(V))^{S_d})^*.\] 
The first of the above isomorphisms is given by the standard {\it multilinearization process}:  it assigns to the homogeneous form $f\in \coor(\mc^{n\times n})_d$ the linear functional  
$\hat f\in \mathrm{T}^d(\een_\mc(V))^*$, characterized by the property that for 
$M_1,\dots,M_d\in \een_\mc(V)$, 
$\hat f(M_1\otimes \cdots \otimes M_d)$ is the value at $(M_1,\dots,M_d)$  
of  the multilinear component of the function $\een_\mc(V)\times \cdots \times \een_\mc(V)\to \mc$, $(M_1,\dots,M_d)\mapsto f(M_1+\cdots+M_d)$. 
The symmetric group $S_d$ acts on $\mathrm{T}^d(\een_\mc(V))$ by permuting the tensor 
components. Clearly  $\hat f(M_1\otimes \cdots \otimes M_d)=\hat f(M_{s(1)}\otimes 
\cdots \otimes M_{s(d)})$ holds for any permutation $s\in S_d$, so 
$\hat f$ belongs to the subspace $(\mathrm{T}^d(\een_\mc(V))^*)^{S_d}$. 
Moreover, $f\mapsto \hat f$ is an isomorphism, because for any $h\in (\mathrm{T}^d(\een_\mc(V))^*)^{S_d}$ we have that $h=\hat f$ where $f(M)=\frac{1}{d!}h(M\otimes \cdots \otimes M)$. The second isomorphism is given by restriction of an $S_d$-invariant linear functional 
on $\mathrm{T}^d(\een_\mc(V))$ to the subspace  $\mathrm{T}^d(\een_\mc(V))^{S_d}$ of symmetric tensors. 
This sends $\hat f\mapsto \tilde f$. 
Observe finally $\mathrm{T}^d(\een_\mc(V))^{S_d}$ is spanned by $M\otimes \cdots \otimes M$ 
as $M$ ranges over $\een_\mc(V)$, see for example \cite[Lemma 9.1.1]{procesi}. 
So we have $\mathrm{T}^d(\een_\mc(V))^{S_d}=\uuu_d$ by definition of the latter. 
\end{proof} 

For a subset $S\subset \een_\mc(\mathrm{T}^d(V))$ denote by 
$Z(S)$ the centralizer of $S$ in the algebra $\een_\mc(\mathrm{T}^d(V))$; clearly 
$Z(S)$ is a subalgebra in $\een_\mc(\mathrm{T}^d(V))$. 
The endomorphism algebra 
$\een_G(\mathrm{T}^d(V))=Z(\uuu(G)_d)$ is generated by the algebra 
$\een_{GL(V)}(\mathrm{T}^d(V))=Z(\uuu_d)$ and a single 
linear transformation $B\in\een_\mc(\mathrm{T}^d(V))$ given by 
\[B(v_1\otimes \cdots \otimes v_d)=\langle v_1\mid v_2\rangle 
\sum_{i=1}^n e_i\otimes e_i\otimes v_3\otimes \cdots v_d,\]  
see for example \cite[Section 11.3.2]{procesi}, where this is deduced from the first fundamental theorem for vector invariants of the orthogonal group (the algebra 
$Z(\uuu(G)_d)$ is called {\it Brauer centralizer algebra}).  
Since $\uuu_d$ and $\uuu(G)_d$ are semisimple subalgebras of $\een_\mc(\mathrm{T}^d(V))$, by the double centralizing theorem for semisimple algebras (see for example \cite[Theorem 6.2.5]{procesi}), one recovers them as the centralizer of their centralizer. 
Thus we have 
\begin{align*}\uuu(G)_d=Z(Z(\uuu(G)_d))=Z(\{B\}\cup Z(\uuu_d))
\\ =Z(\{B\})\cap 
Z(Z(\uuu_d))=\uuu_d\cap Z(\{B\}).
\end{align*}
This can be reformulated as follows: 

\begin{proposition}\label{prop:brauer}  
The kernel of the natural surjection $\uuu_d^*\twoheadrightarrow \uuu(G)_d^*$ 
(given by restriction of linear functionals on $\uuu_d$ to the subspace $\uuu(G)_d$) 
is spanned by the restrictions to $\uuu_d$ of the coordinate functions of the map 
\[\een_\mc(\mathrm{T}^d(V))\to \een_\mc(\mathrm{T}^d(V)),\qquad 
Y\mapsto BY-YB.\] 
\end{proposition} 

\begin{proofofthm}~\ref{thm:ideal_generators}. 
Denote by $Q$ the ideal in $\coor(\mc^{n\times n})$ generated by the $n^2+n-2$ quadratic polynomials in the statement. By definition of $G$ they vanish on $G$, and since they are homogeneous, they vanish on $\cone(G)$ as well. So $Q\subseteq \ideal(\cone(G))$.  
The ideal $Q$ is generated by homogeneous elements. 
Being the vanishing ideal of a cone, $\ideal(\cone(G))$ is also a homogeneous ideal. 
Denote by $Q_d$, $\ideal(\cone(G))_d$ their degree $d$ homogeneous components. 
It remains to prove that $Q_d=\ideal(\cone(G))_d$ for all positive integers $d$. 
By Lemma~\ref{lemma:schur-weyl} the isomorphism 
$\mu:\coor(\mc^{n\times n})_d\to \uuu_d^*$, $f\to \tilde f$ maps $\ideal(\cone(G))_d$ to 
the kernel of the natural surjection $\uuu_d^*\twoheadrightarrow \uuu(G)_d^*$. 
Therefore by Poposition~\ref{prop:brauer} it is sufficient to show that 
the restriction to $\uuu_d$ of the coordinate functions of 
the map $\kappa:\een_\mc(\mathrm{T}^d(V))\to \een_\mc(\mathrm{T}^d(V))$, $Y\mapsto BY-YB$ are contained 
in the space $\mu(Q_d)$. 
Now we compute the restrictions to $\uuu_d$ of the coordinate functions of 
$\kappa$ with respect to the basis $E_{i_1j_1}\otimes \cdots \otimes E_{i_dj_d}$ 
of $\een_\mc(\mathrm{T}^d(V))\cong \mathrm{T}^d(\een_\mc(V))$,  
where 
\[E_{ij}(e_k)=\begin{cases} e_i\mbox{ for }k=j\\
0\mbox{ for }k\neq j. 
\end{cases} \]
In order to do so we need  to point out that under the identification 
$\een_\mc(\mathrm{T}^d(V))\cong \mathrm{T}^d(\een_\mc(V))$ the element $B$ corresponds 
to 
\[B=\sum_{i,j=1}^nE_{ij}\otimes E_{ij}\otimes \mathrm{id}_V\otimes \cdots 
\otimes \mathrm{id}_V. \]
For $M=\sum_{i,j=1}^nM_{ij}E_{ij}\in \een_\mc(V)$  
we have 
\begin{align*}M\otimes M \sum_{k,l=1}^nE_{kl}\otimes E_{kl}
=\sum_{i,j,s,t,k,l}M_{ij}M_{st}E_{ij}E_{kl}\otimes E_{st}E_{kl}
\\=\sum_{i,k,s,l}M_{ik}M_{sk}E_{il}\otimes E_{sl}
=\sum_{i,s,l}(\sum_{k=1}^nM_{ik}M_{sk})E_{il}\otimes E_{sl}
\end{align*}
and 
\begin{align*}(\sum_{k,l=1}^nE_{lk}\otimes E_{lk})M\otimes M 
=\sum_{i,j,s,t,k,l}M_{ji}M_{ts}E_{lk}E_{ji}\otimes E_{lk}E_{ts}
\\=\sum_{i,k,s,l}M_{ki}M_{ks}E_{li}\otimes E_{ls}
=\sum_{i,s,l}(\sum_{k=1}^nM_{ki}M_{ks})E_{li}\otimes E_{ls}
\end{align*}
This shows that the restriction to $\uuu_d$ of the coordinate function 
of the map $\kappa$ corresponding to $E_{il}\otimes E_{sl}\otimes E_{i_3j_3}\otimes \cdots \otimes E_{i_dj_d}$ 
is the image under $\mu$ of  
\[\frac{1}{d!}\cdot \begin{cases} x_{i_3j_3}\cdots x_{i_dj_d}\sum_{k=1}^nx_{ik}x_{sk}\  \mbox{ when }i\neq s
\\x_{i_3j_3}\cdots x_{i_dj_d}\sum_{k=1}^n(x_{ik}^2-x_{kl}^2)\ \mbox{ when }i=s\end{cases}, 
\] 
the restriction to $\uuu_d$ of the coordinate function 
of the map $\kappa$ corresponding to $E_{li}\otimes E_{ls}\otimes E_{i_3j_3}\otimes \cdots \otimes E_{i_dj_d}$ for $i\neq s$ is the image under $\mu$ of 
\[\frac{1}{d!}x_{i_3j_3}\cdots x_{i_dj_d}\sum_{k=1}^nx_{ki}x_{ks},\] 
and the remaining coordinate functions of 
$\kappa$ restrict to the zero map on $\uuu_d$. 
Since $\sum_{k=1}^nx_{ik}x_{sk}$, $\sum_{k=1}^nx_{ki}x_{ks}$ for $i\neq s$ and  
$\sum_{k=1}^n(x_{ik}^2-x_{kl}^2)$ belong to $Q$, this shows that 
the restriction to $\uuu_d$ of each of the coordinate functions of $\kappa$ is contained 
in $\mu(Q_d)$, and we are done. 
\end{proofofthm}

%%%%%%%%%%%%%%%%%%%%%%%%%%%%%%%%%%%%%%%%%
\section{Gr\"obner basis for the vanishing ideal of $\cone(G)$}\label{sec:groebner}

\subsection{Gr\"obner basis for $\ideal(\cone(\OR(3)))$} \label{subsec:groebner} 

A Gr\"obner basis of the vanishing ideal of  the cone of $\OR(3,\beta)$ with respect to the 
degree reverse lexicographic ordering of the monomials, induced by the ordering 
$x_{11}>x_{12}>x_{13}>x_{21}>x_{22}>x_{23}>x_{31}>x_{32}>x_{33}$ 
 of the variables is the following: 
 \begin{align*} 
 x_{12} x_{21} x_{31} - 2 x_{11} x_{22} x_{31} + x_{11} x_{21} x_{32}, 
 \\ x_{12} x_{21} x_{32} - 2 x_{11} x_{22} x_{32} - 2 x_{11} x_{21} x_{33}, 
 \\ x_{11} x_{23} x_{32} - x_{12} x_{21} x_{33}, 
 \\ x_{12} x_{23} x_{32} - 2 x_{12} x_{22} x_{33} - 2 x_{11} x_{23} x_{33}, 
 \\ x_{13} x_{23} x_{32} - 2 x_{13} x_{22} x_{33} + x_{12} x_{23} x_{33}, 
 \\ x_{12}^{2} + 2 x_{11} x_{13}, 
 \\ x_{13} x_{21} + x_{12} x_{22} + x_{11} x_{23}, 
 \\ x_{21}^{2} + 2 x_{11} x_{31}, 
 \\ x_{21} x_{22} + x_{12} x_{31} + x_{11} x_{32}, 
 \\ x_{22}^{2} - x_{13} x_{31} + x_{12} x_{32} - x_{11} x_{33}, 
 \\ x_{21} x_{23} - x_{12} x_{32}, 
 \\ x_{22} x_{23} + x_{13} x_{32} + x_{12} x_{33}, 
 \\  x_{23}^{2} + 2 x_{13} x_{33}, 
 \\ x_{23} x_{31} + x_{22} x_{32} + x_{21} x_{33}, 
 \\ x_{32}^{2} + 2 x_{31} x_{33}
 \end{align*}

 The leading monomials of the elements in the above Gr\"obner basis are 
 \begin{align*}x_{32}^2,\  x_{23}x_{31},\ x_{23}^2,\  x_{22}x_{23},\ x_{21}x_{23},\ x_{22}^2,\  x_{21}x_{22},\ x_{21}^2,\  x_{13}x_{21},\   x_{12}^2,\ 
\\  x_{13} x_{23} x_{32},\ x_{12} x_{23} x_{32},\ x_{11} x_{23} x_{32},\ x_{12} x_{21} x_{32},\ x_{12} x_{21} x_{31}.\end{align*}

 For a set $\{a,b,\dots\}$ of variables denote by $\langle a,b,\dots \rangle$ the set of monomials involvig only the variables $a,b,\dots$. 
   The standard monomials are the following: 
  \begin{align*} 
  \langle x_{11},x_{13},x_{31},x_{33}\rangle 
  \sqcup x_{12}\langle x_{11},x_{13},x_{31},x_{33}\rangle\sqcup  x_{22}\langle x_{11},x_{13},x_{31},x_{33}\rangle 
  \sqcup x_{32}\langle x_{11},x_{13},x_{31},x_{33}\rangle 
  \\  x_{12}x_{32}\langle x_{11},x_{13},x_{31},x_{33}\rangle \sqcup x_{12}x_{22}  \langle x_{11},x_{13},x_{31},x_{33}\rangle 
  \sqcup   x_{22}x_{32} \langle x_{11},x_{13},x_{31},x_{33}\rangle 
  \\ \sqcup x_{12}x_{22}x_{32} \langle x_{11},x_{13},x_{31},x_{33}\rangle 
  \sqcup x_{21} \langle x_{11},x_{31},x_{33}\rangle \sqcup x_{23}\langle x_{11},x_{13},x_{33}\rangle 
 \\  \sqcup x_{12}x_{23} \langle x_{11},x_{13},x_{33}\rangle
  \sqcup x_{21}x_{32}\langle x_{13},x_{31},x_{33}\rangle
\sqcup  x_{12}x_{21}\langle x_{11},x_{33}\rangle \sqcup x_{23}x_{32}\langle x_{33}\rangle 
  \end{align*}  
  A form of the Hilbert series $\mathrm{H}(\coor(\cone(G));t)$ reflecting the above list is 
\[H(\coor(\cone(G));t)=\frac{1+3t+3t^2+t^3}{(1-t)^4}+\frac{2t+2t^2}{(1-t)^3}+\frac{t^2}{(1-t)^2}+\frac{t^2}{1-t}.\]

 The quadratic elements in this Gr\"obner basis form a minimal generating set of the ideal  $\ideal(\cone(\OR(3))$, since    
 (up to our coordinate change) they form a $\mc$-vector space basis in the linear span of the  generators of $\ideal(\cone(\OR(3))$  
 given in Theorem~\ref{thm:ideal_generators}.  
  
 \begin{proposition} Regardless of the chosen term order in $\coor(\mc^{n\times n})$, the ideal $\mathcal{I}(\cone(\OR(3)))$ has no quadratic Gr\"obner basis. 
 \end{proposition} 
 
 \begin{proof} Suppose to the contrary that there is a term order such that   $\mathcal{I}(\cone(\OR(3)))$ has a quadratic Gr\"obner basis. 
 Then the graded algebra $\coor(\cone(\OR(3)))$ is Koszul, see \cite{froberg}, \cite{anick}.  This contradicts to 
 Proposition~\ref{prop:koszul}. 
 \end{proof} 
 
 \subsection{Gr\"obner basis for $\ideal(\cone(\SP(4)))$}

We do not know if $\coor(\cone(G))$ for $G=\OR(4)$ or $G=\SP(4)$ is Koszul or not. 
In contrast with $\coor(\cone(\OR(3)))$, for these groups $G$  the first $50$ coefficients of $1/\mathrm{H}(\coor(\cone(G));-t)$ are non-negative. 
On the other hand, the quadratic generators for $\ideal(\cone(G))$ do not constitute a Gr\"obner basis with respect to the 
degree reverse lexicographic monomial ordering induced by the ordering 
$x_{11}>x_{12}>x_{13}>x_{14}>x_{21}>x_{22}>x_{23}>x_{24}>x_{31}>x_{32}>x_{33}>x_{34}>x_{41}>x_{42}>x_{43}>x_{44}$ 
of the variables. 
For example, the Gr\"obner basis for $\ideal(\cone(\SP(4))$ has two cubic elements, in addition to $10$ quadratic elements: 
\begin{align*} 
x_{11} x_{32} x_{41} - x_{14} x_{33} x_{41} + x_{13} x_{34} x_{41} - x_{11} x_{31} x_{42} + x_{14} x_{31} x_{43} - x_{13} x_{31} x_{44}, 
\\ x_{21} x_{32} x_{41} - x_{24} x_{33} x_{41} + x_{23} x_{34} x_{41} - x_{21} x_{31} x_{42} + x_{24} x_{31} x_{43} - x_{23} x_{31} x_{44}, 
\\ x_{12} x_{21} - x_{11} x_{22} - x_{34} x_{43} + x_{33} x_{44}, \quad x_{13} x_{21} - x_{11} x_{23} + x_{33} x_{41} - x_{31} x_{43}, 
\\ x_{14} x_{21} - x_{11} x_{24} + x_{34} x_{41} - x_{31} x_{44}, \quad x_{13} x_{22} - x_{12} x_{23} + x_{33} x_{42} - x_{32} x_{43}, 
\\ x_{14} x_{22} - x_{12} x_{24} + x_{34} x_{42} - x_{32} x_{44}, \quad x_{14} x_{23} - x_{13} x_{24} - x_{32} x_{41} + x_{31} x_{42}, 
\\ x_{12} x_{31} - x_{11} x_{32} + x_{14} x_{33} - x_{13} x_{34}, \quad x_{22} x_{31} - x_{21} x_{32} + x_{24} x_{33} - x_{23} x_{34}, 
\\ x_{12} x_{41} - x_{11} x_{42} + x_{14} x_{43} - x_{13} x_{44}, \quad x_{22} x_{41} - x_{21} x_{42} + x_{24} x_{43} - x_{23} x_{44}
\end{align*}

%%%%%%%%%%%%%%%%%%%%%%%%%%%%%%%%%%%%%%%%%%%%

\end{document}